\documentclass[reqno]{amsart}
\usepackage{epsfig,color}

\usepackage{graphicx} 					
\usepackage{booktabs} 					
\usepackage{array}    						
\usepackage[parfill]{parskip}			
\usepackage{enumitem} 
\usepackage{cite} 
\usepackage[autostyle]{csquotes} 

\usepackage{lipsum} 
\usepackage{multirow} 
\usepackage{graphics} 
\usepackage[all]{xy} 
\usepackage{graphics,caption}
\usepackage{caption}
\usepackage{subcaption}

\usepackage{tikz}
\usetikzlibrary{patterns}
\usetikzlibrary{arrows.meta}
\usetikzlibrary{bending}
\usepackage{tikz-cd} 
\usetikzlibrary{calc, intersections}
\usepackage{mathrsfs}
\usepackage{amssymb}
\usepackage{amsmath}
\usepackage{amsthm} 
\usepackage{mathtools} 
\usepackage[retainorgcmds]{IEEEtrantools}
\usepackage{changepage}
\usepackage{hyperref}

\theoremstyle{plain}
\newtheorem{thm}{Theorem}[section]
\newtheorem{lemma}[thm]{Lemma}
\newtheorem{cor}[thm]{Corollary}
\newtheorem{prop}[thm]{Proposition}

\theoremstyle{definition}
\newtheorem{definition}[thm]{Definition}
\theoremstyle{remark}
\newtheorem*{rmk}{Remark}
\theoremstyle{definition}

\newcommand{\Z}{\mathbb Z}
\newcommand{\R}{\mathbb R}

\newcommand{\N}{\mathbb N}

\newcommand{\Gsigma}{\mathcal{G}(\sigma) }
\newcommand{\Gosigma}{\mathcal{G}^0(\sigma) }
\newcommand{\Gssigma}{\mathcal{G}^*(\sigma) }
\newcommand{\Gtsigma}{\mathcal{G}^2(\sigma) }
\newcommand{\Gs}{\mathcal{G}(\tilde{S}) }

\newcommand{\Gss}{\mathcal{G}^*(\tilde{S}) }
\newcommand{\Gts}{\mathcal{G}^2(\tilde{S}) }
\newcommand{\supp}{\mathrm{supp}\, L_\sigma}
\newtheorem*{thmmain}{Theorem~\ref{main theorem}}
\newtheorem*{cormain}{Corollary~\ref{main corollary}}

\begin{document}

\title[Liouville current of quadratic differential metrics]{The Liouville current of holomorphic quadratic differential metrics}
\author{Jiajun Shi}
\address{Mathematisches Institut, Ruprecht-Karls Universität Heidelberg, 69120 Heidelberg, Germany}
\email{jshi@mathi.uni-heidelberg.de}

\begin{abstract}
    In this paper we prove that on a closed oriented surface, flat metrics coming from holomorphic quadratic differentials can be distinguished from other flat cone metrics by the Liouville current.
\end{abstract}

\maketitle

\section{Introduction}

Geodesic currents were introduced by Bonahon in \cite{bonahon1988geometry} to study compactifications of Teichmüller space. Bonahon constructed for every hyperbolic metric a special geodesic current, called the \textit{Liouville current}, and gave an embedding of Teichmüller space into the space of geodesic currents. In \cite{otal1990spectre}, Otal generalized this result and constructed an embedding of the set of negatively curved metrics. After that, it was furtherly generalized to the set of singular metrics, first by Hersonsky and Paulin \cite{hersonsky1997rigidity} to the set of negatively curved cone metrics and then by Bankovic and Leininger \cite{bankovic2018marked} to the set of flat cone metrics. 

A specific class of flat cone metrics arise from holomorphic quadratic differentials. A \textit{holomorphic quadratic differential} $q$ on a Riemannian surface $X$ is a holomorphic section of $T^*X\otimes T^*X$, where $T^*X$ is the holomorphic cotangent bundle. It naturally gives the surface a flat cone structure. Let $QFlat(S)$ be the set of flat cone metrics coming from a holomorphic quadratic differential on $S$. Our main theorem distinguishes $QFlat(S)$ among flat cone metrics from the perspective of Liouville currents.
\begin{thmmain}
Let $(S,\sigma)$ be a closed oriented flat cone surface of genus at least 2. Then $\sigma$ lies in $QFlat(S)$ if and only if the Liouville current is integral and totally simple (see Definition \ref{def_integral} and \ref{def_totally_simple}).
\end{thmmain}

We first show that a flat cone metric is in $QFlat(S)$ if and only if the cone angles are integer multiples of $\pi$ and the holonomy is $\pm\text{Id}$, see Theorem \ref{angle and holonomy}. Thus, our task is to read the information about cone angle and holonomy from the Liouville current. Bankovic and Leininger used a novel method to prove that the cone angles of two flat cone metrics are the same if their Liouville current are the same. We modify this method and manage to compute the cone angle from the Liouville current. The main difficulty is about the holonomy, because the Liouville current is only related to the set of nonsingular geodesics, while the holonomy is about the parallel transport for all closed geodesics. Here we use the trick that the existence of a simple dense geodesic implies the holonomy is $\pm\text{Id}$, so the problem is transformed into how to ensure the existence of a simple dense geodesic. In general, the closure of a simple geodesic admits a foliation structure, which corresponds to a geodesic lamination. The complement of a geodesic lamination is fully classified (Proposition \ref{principal}). We use this to add certain restrictions to Liouville currents to make sure that there exists a geodesic such that the complement of its closure is empty (Theorem \ref{dense}), which implies that it is dense in the surface. These restrictions are summarized by the word "totally simple" in Definition \ref{def_totally_simple}. Moreover, using the property of totally simpleness, we obtain the following interesting result, which roughly says that almost every nonsingular geodesic will be dense in the surface.
\begin{cormain}
    Given a surface with a quadratic differential and any point $p$, there must exist a nonsingular geodesic passing through $p$ which is dense in the surface. Moreover, if we parameterize the space of nonsingular geodesics passing through $p$ by the angle of tangent vector at $p$, then almost every such geodesic will be dense in the surface.
\end{cormain}

The paper is organized as follows. Section 2 outlines the basic properties of flat cone metric and quadratic differentials. In Section 3, we introduce and modify the chain method of \cite{bankovic2018marked} to compute the cone angle of a flat cone metric. In Section 4, we list some useful facts about lamination and foliation, and use them to prove the existence of a dense flat geodesic for integral and totally simple Liouville current. The proof relies on the relations between lamination and foliation, which involves both flat metrics and hyperbolic metrics on a surface. Finally, Section 5 presents how we could determine the holonomy condition using a dense flat geodesic.

\section{Preliminary}

\subsection{Flat cone metric}
A \textit{Euclidean cone} is obtained by gluing in a cyclic pattern a finite number of sectors in the Euclidean plane along the edges and the vertices are glued together. The cone angle of a Euclidean cone is the sum of  sector angles. If the cone angle is $2\pi$, then the cone is just a local Euclidean plane.

A closed oriented surface $(S,\sigma)$ is called a \textit{flat cone surface} if 
\begin{enumerate}
    \item[(1)] every point $p\in S$ has a neighborhood which is isometric to a Euclidean cone of cone angle $\theta(p)$ with $p$ mapped to the cone point
    \item[(2)] $\theta(p)\ge 2\pi$ and equality holds except for only finitely many points.
\end{enumerate}
A point $p\in S$ is called a \textit{cone point} or \textit{singular point} if $\theta(p)> 2\pi$. By definition, cone points form a discrete subset of $S$.

A \textit{geodesic} is a curve locally realizing the minimal distance. For a flat cone surface $(S,\sigma)$, a geodesic $\gamma$ must be a concatenation of Euclidean straight lines, where two adjacent lines meet at a cone point. We can give $\gamma$ a transverse orientation so that we could talk about the positive and negative sides. $\gamma$ makes two angles when it passes through a cone point, one on the positive side and the other on the negative side. Since $\gamma$ is a geodesic, both angles must be no smaller than $\pi$. Conversely, if a curve is a concatenation of some Euclidean lines and make an angle at least $\pi$ at each cone point, then it is a geodesic. A geodesic is called \textit{nonsingular} if it contains no cone points. A geodesic segment is called a \textit{saddle connection} if it connects two cone points and contains no other cone points in the interior.

By \cite[Theorem II.5.5]{bridson2013metric}, $(S,\sigma)$ is a locally CAT(0) space. Thus, its universal covering with induced flat cone structure $(\tilde{S},\sigma)$ is a CAT(0) space. By Cartan-Hadamard theorem, $\tilde{S}$ is homeomorphic to $\R^2$. Two geodesic rays $c_1,c_2:[0,+\infty)\rightarrow \tilde{S}$ are called asymptotic if there is a constant $K$ such that $d(c_1(t),c_2(t))\le K$ for all $t$. Let $S^1_\infty$ be the boundary of $\tilde{S}$, namely the set of geodesic rays up to asymptotic equivalence. Let $\Gsigma$ denote the set of biinfinite geodesics in $\tilde{S}$ equipped with the compact-open topology. Let $\Delta \subset S^1_\infty\times S^1_\infty$ be the diagonal subset, and $\mathcal{G}(\tilde{S}) = (S^1_\infty \times S^1_\infty \setminus \Delta) / (x,y)\sim(y,x)$, i.e., unordered pair of boundary points. There is a natural continuous map 
\[\partial:\Gsigma\to \Gs  \]
which sends a geodesic to two equivalent classes of geodesic rays.

\begin{lemma}\label{closed}
For a flat cone surface $(S,\sigma)$, the map $\partial$ is a closed surjective map.
\end{lemma}
\begin{proof}
For surjectivity, see \cite[Theorem II.9.33]{bridson2013metric}. To prove this map is closed, we need to show that for a sequence of geodesics $\{\gamma_n\}$ in $\Gsigma$ whose endpoints converge to some $(x,y)$, there exist a subsequence of $\{\gamma_n\}$ converging to a geodesic with endpoints $(x,y)$. 
Since the endpoints of $\gamma_n$ converge, all $\gamma_n$ will pass through a common compact ball in $\tilde{S}$, see \cite[\textsection 4.14]{ballmann1985manifolds}. Thus we can apply Arzela-Ascoli theorem and find a subsequence of $\{\gamma_n\}$ converging to a geodesic segment $\delta$ in this compact ball. If $\delta$ could be extended to a nonsingular geodesic $\gamma$, we are done. Otherwise, suppose $\delta$ meets a cone point. Take another compact ball containing this cone point and apply Arzela-Ascoli theorem again to find a subsequence whose limit extends $\delta$ through the cone point. Since there are only countably many cone points, we could use the diagonal method and find a subsequence converging to a bi-infinite geodesic $\gamma$ containing $\delta$ as a subsegment. The continuity of $\partial$ implies that the endpoints of $\gamma$ are exactly $(x,y)$.
\end{proof}

It is well-known that the map $\partial$ is a homeomorphism for hyperbolic metrics. But it's not true for flat cone metrics. In the following, we will discuss when it is not injective, namely, two different geodesics are asymptotic in both directions.

Suppose $X$ is hoemomorphic to $\R^2$ and is equipped with a flat cone metric. Take two asymptotic rays $\alpha_1,\alpha_2:[0,\infty)\to X$ with parameter $t$. They could not intersect transversely because $d(\alpha_1(t),\alpha_2(t))$ is a convex function for CAT(0) spaces. Either there is no intersection, or they pass through the same cone point and share a common subray afterwards. In the latter case, these two rays are called \textit{cone-asymptotic}.

\begin{figure}[h]
    \centering
    \begin{tikzpicture}
    \draw (-1,1) -- (1,0) -- (3.5,0);
    \draw (-1,-1) -- (1,0);
    \filldraw (1,0) circle (2pt);
    \node[above] at (0.5,0.5) {$\alpha_1$};
    \node[below] at (0.5,-0.5) {$\alpha_2$};
    \end{tikzpicture}
    \caption{Cone asymptotic geodeisc rays}
    \label{asymptotic}
\end{figure}
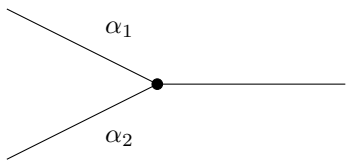

\begin{lemma}\label{asymptotic rays}
If two geodesic rays $\alpha_1,\alpha_2$ in $X$ are asymptotic, then they must be in one of the following cases:
\begin{enumerate}
    \item[(1)]  $\alpha_1$ and $\alpha_2$ are cone-asymptotic.
    \item[(2)] $\alpha_1$ and $\alpha_2$ bound a flat half strip.
    \item[(3)] one ray approaches the other by connecting infinitely many cone points.
\end{enumerate}
\end{lemma}
\begin{proof}
Suppose two rays are not cone-asymptotic. Let $\alpha_2$ is parametrized by $t\in[0,+\infty)$. Denote by $\beta_t$ the geodesic orthogonal to $\alpha_2$ at $\alpha_2(t)$. Since there are only countably many cone points, there is a unique $\beta_t$ for almost all $t$. Note that for large enough $t$, $\beta(t)$ will definitely intersect $\alpha_1$ at some point. Let $\theta_t$ be the angle formed by $\beta_t$ and $\alpha_1$ as is shown in Figure \ref{varying}. If $\theta_t>\frac{\pi}{2}$ for some $t$, then $\alpha_1$ and $\alpha_2$ cannot limit to the same endpoint by using comparison law in CAT(0) spaces.

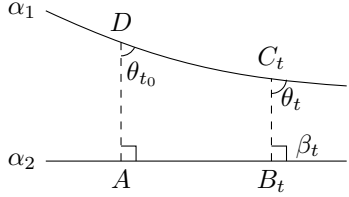
\begin{figure}[h]
    \centering
    \begin{tikzpicture}
    \draw (-2,2)..controls (-0.5,1.3) and (0.5,1.1).. (2,1);
    \draw (-2,0)--(2,0);
    \draw[dashed] (-1,0)--(-1,1.6);
    \draw[dashed] (1,0)--(1,1.1);
    \draw (-1,0.2)--(-0.8,0.2)--(-0.8,0) (1,0.2)--(1.2,0.2)--(1.2,0);
    \draw (1,0.9) arc [radius=0.2,start angle=-90,end angle=-5];
    
    \node[left] at(-2,0) {$\alpha_2$};
    \node[left] at (-2,2) {$\alpha_1$};
    \node[below right] at (1,1.1) {$\theta_t$};
    \node[right] at (1.2,0.2) {$\beta_t$};
    \node[below] at (-1,0) {$A$};
    \node[below] at (1,0) {$B_t$};
    \node[above] at (1,1.1) {$C_t$};
    \node[above] at (-1,1.6) {$D$};
    
    \draw (-1,1.4) arc [radius=0.2,start angle=-90,end angle=-25];
    \node[below right] at (-1.05,1.45) {$\theta_{t_0}$};
    \end{tikzpicture}
    \caption{Two asymptotic rays and $\theta_t$}
    \label{varying}
\end{figure}

Now assume that $\theta_t\le\frac{\pi}{2}$ for all $t$. Fix one $\beta_{t_0}$ and take a $\beta_t$ for some $t>t_0$. $\beta_{t_0},\beta_t,\alpha_1,\alpha_2$ enclose a quadrilateral $AB_tC_tD$. Glue it with another identical quadrilateral along vertices and edges, we get a flat cone surface homeomorphic to a sphere. The cone points of this sphere consist of the original cone points inside the polygon, cone points in the edges and the vertices $A,B_t,C_t,D$.

Suppose there are only finitely many cone points in both $\alpha_1$ and $\alpha_2$. Then we can always replace $\alpha_i$ by shorter subrays and assume there is no cone point in both rays. Denote the cone points inside the polygon $AB_tC_tD$ by $p_1,\ldots,p_{N_t}$. Each cone point $p_i$ contributes a concentrated curvature $\kappa_i=2\pi-\theta(p_i)<0$, $\kappa_A=\kappa_{B_t}=\pi, \kappa_{C_t}=2\pi-2\theta_t, \kappa_{D}=2\theta_{t_0}$. Apply Gauss-Bonnet theorem, we get
\[\theta_t=\theta_{t_0}+\sum_{i=1}^{N_t}\kappa_i\]

Thus $\theta_t$ is a nondecreasing step function with respect to $t$. Since $\theta_t$ is bounded above by $\frac{\pi}{2}$, $\theta_t$ converges as $t$ goes to infinity and therefore, $\kappa_i$ converges to $0$. Note that these cone points in the polygon come from the lift of finitely many cone points in the surface $S$, so $\kappa_i$ could only take finitely many values, which forces $\kappa_i=0$ for almost all $i$. Therefore, up to taking subrays, $\alpha_1$ and $\alpha_2$ bound a flat region, and it must be a flat half strip. In particular, $\theta_t=\frac{\pi}{2}$ for all large enough $t$. 

The remaining case is that there are infinitely many cone points in at least one of the geodesic rays $\alpha_1,\alpha_2$. For a cone point $p$ in the edges of the polygon, it could contribute a curvature arbitrarily close to 0, which makes it possible that $\theta_t$ converges to $\frac{\pi}{2}$ and $\theta_t\ne \frac{\pi}{2}$.
\end{proof}

Note that the third case happens if and only if the cone angle could take infinitely many different values. This is impossible if $X$ is the universal covering of a flat cone surface. We obtain the following corollary.

\begin{cor}\label{flat_strip}
Suppose $\tilde(S)$ is the universal covering of a flat cone surface. If two geodesics $\gamma_1,\gamma_2$ in $\tilde{S}$ limit to the same pair of endpoints, then they must bound a flat strip, which is isometric to $\R\times [a,b]$ for some interval $[a,b]$.
\end{cor}
\begin{proof}
Since $\gamma_1$ and $\gamma_2$ are asymptotic in both directions, it's just the combination of three cases in Lemma \ref{asymptotic rays}. Moreover, because of the statement above, the only possible combination is (1)-(1), (1)-(2) or (2)-(2). Case (1)-(1) is impossible because there is at most one geodesic between two cone points. (1)-(2) is also impossible because there is a one-to-one correspondence between $S^1_\infty$ and geodesic rays starting from a cone point. Thus (2)-(2) is the only possible case, namely these two geodesics must bound a flat strip.
\end{proof}

Let $\mathcal{G}^0(\sigma)\subset\Gsigma$ be the set of nonsingular geodesics, and $\mathcal{G}^*(\sigma)=\overline{\Gosigma}$. In the following chapters, we will focus on $\Gssigma$ and its subsets.

\begin{prop}\label{saddle}
A geodesic $\gamma$ is contained in $\Gssigma$ if and only if
\begin{enumerate}
    \item[(1)] $\gamma$ makes an angle $\pi$ on one side at each cone point.
    \item[(2)] ordering the cone points linearly along $\gamma$, the side where the angle is $\pi$ switches at most once.
\end{enumerate}
\end{prop}

\begin{proof}
$\Rightarrow$: Suppose $\{\gamma_n\}$ is a sequence of nonsingular geodeiscs approximating $\gamma$. Then at each cone point, there is a subsequence approaching $\gamma$ from either the positive side or the negative side. The cone angle at this side is exactly $\pi$. 

Now suppose the side where the angle is $\pi$ switches at least twice. Then for sufficiently large $n$, as is shown in Figure \ref{twice}, $\gamma_n$ crosses $\gamma$ at least twice, which is impossible for a CAT(0) space.
\begin{figure}[h]
    \centering
    \begin{tikzpicture}
    \draw (-4,1).. controls (-1,0) and (1,0).. (4,1);
    \draw[very thick] (-4,0.5)-- (-2,0)-- (0,1) -- (2,0) -- (4,0.5);
    \draw (-2.194,0.05) arc [radius=0.2, start angle=166, end angle=27];
    \draw (-0.18,0.91) arc [radius=0.2, start angle=-153, end angle=-27];
    \draw (2.194,0.05) arc [radius=0.2, start angle=14, end angle=153];
    \node [above] at (-2.2,0.15) {$\pi$};
    \node [below] at (0,0.8) {$\pi$};
    \node [above] at (2.2,0.15) {$\pi$};
    \node at (-3,1) {$\gamma_n$};
    \node[below] at (-4,0.5) {$\gamma$};
    
    \node[below] at((-2,0) {$\xi_1$};
    \node[above] at(0,1) {$\xi_2$};
    \node[below] at (2,0) {$\xi_3$};
    
    \filldraw (-2,0) circle (2pt) (0,1) circle (2pt) (2,0) circle (2pt);
    \end{tikzpicture}
    \caption{The side where the angle is $\pi$ switches twice when $\gamma$ goes from $\xi_1$ to $\xi_3$.}
    \label{twice}
\end{figure}
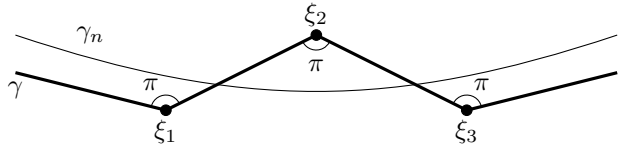

$\Leftarrow$: It suffices to construct a sequence of approximating nonsingular geodesics. Suppose $\gamma(t)$ is a geodesic and $\{ \xi_i | -\infty\le i \le +\infty \}$ is the set of linearly ordered cone points on $\gamma$. Suppose $\gamma$ make an angle $\pi$ at the positive side of $\xi_i$ for $i\le N$ and at the negative side of $\xi_i$ for $i\ge N+1$. Take the saddle connection joining $\xi_N$ and $\xi_{N+1}$ (if $N=\pm\infty$, just take any of the saddle connections) and choose a nonsingular point $p$ on it. It´s easy to construct a geodesic which pass through $\gamma$ at $p$ from the positive side to the negative side and form an angle $\theta$ with $\gamma$. Find a sequence of such geodesics $\gamma_n$ such that $\lim \theta_n=0$. Since there are only countably many cone points in the universal covering, we may assume all $\gamma_n$ are nonsingular by carefully choosing the angles $\theta_n$. It´s clear from the construction that such a sequence will converge to $\gamma$.
\end{proof}

Let $\Gtsigma$ be the subset of geodesics in $\Gssigma$ which contain at least two cone points. By Proposition \ref{saddle}, there are at most countably many such geodesics containing a fixed saddle connection. Since there are only countably many saddle connections, we conclude that $\Gtsigma$ contains at most countably many geodesics.

Define $\Gss$ (resp. $\Gts$) to be the image of $\Gssigma$ (resp. $\Gtsigma$) under the map $\partial:\Gsigma\to\Gs$. By the argument above, there are only countably many elements in $\Gts$.

\subsection{Holomorphic quadratic differential}\label{holonomy and cone angle}
A \textit{holomorphic quadratic differential} $q$ on a Riemannian surface $X$ is a holomorphic section of $T^*X\otimes T^*X$, where $T^*X$ is the holomorphic cotangent bundle. In a local coordinate $z$, $q=f(z)\,dz^2$ for some holomorphic function $f(z)$, see \cite[\textsection 4.1]{strebel1984quadratic} or \cite[§2.1]{duchin2010length}.

A holomorphic quadratic differential $q$ naturally gives a flat cone metric by constructing standard local coordinates: take a point $p\in X$ and a local coordinate $z$ near $p$ such that $q=f(z)\,dz^2$ and $p$ corresponds to $z=0$.
\begin{enumerate}
    \item[(1)] If $f(0)\ne 0$, then there exists a neighborhood $U$ of $0$ such that $f(z)\ne 0 $ for all $z\in U$. Integrate one branch of $\sqrt{q}$, we get a new local coordinate $\xi$ and $q=d\xi^2$. Moreover, for another local coordiante $\xi'$ satisfying $q=d\xi'^2$, $\xi'=\pm\xi+\mathrm{constant}$. Such coordinates lead to a well-defined Euclidean metric because the local metric is preserved by transition map.
    \item[(2)] If $f(0)=0$ and $0$ is of order $k$, $k\in \N$, then there exists another coordinate $w$ such that $q=w^kdw^2$. Let $\eta = \frac{2}{k+2}w^{\frac{k+2}{2}}$, then locally $q=d\eta^2$. We can see that $p$ is a cone point of cone angle $(k+2)\pi$. See Figure \ref{trajectory} for an example.
    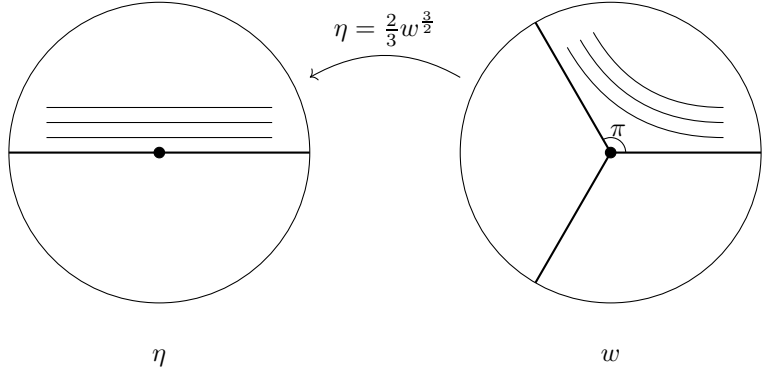
\begin{figure}[h]
    \centering
    \begin{tikzpicture}
    \draw (-3,0) circle (2) (3,0) circle (2);
    
    \draw[thick] (-5,0)--(-1,0);
    \draw (-4.5,0.2)--(-1.5,0.2) (-4.5,0.4)--(-1.5,0.4) (-4.5,0.6)--(-1.5,0.6);
    \filldraw (-3,0) circle (2pt);
    
    \draw[thick] (3,0)--(5,0) (3,0)--(2,1.732) (3,0)--(2,-1.732);
    \filldraw (3,0) circle (2pt);
    \draw (2.4232,1.4) to [out=-60, in=180] (4.5,0.2) (2.5964,1.5) to [out=-60, in=180] (4.5,0.4) (2.7696,1.6) to [out=-60, in=180] (4.5,0.6);
    
    \draw[->] (1,1) to [out = 150, in= 30] (-1,1);
    \node[above] at (0,1.3) {$\eta = \frac{2}{3}w^{\frac{3}{2}}$};
    
    \node[below] at (-3,-2.5) {$\eta$};
    \node[below] at (3,-2.5) {$w$};
    
    \draw (3.2,0) arc [radius=0.2, start angle=0, end angle=120];
    \node[above] at (3.1,0.1) {$\pi$};
    \end{tikzpicture}
    \caption{An example for $k=1$, where the cone angle is $3\pi$.}
    \label{trajectory}
\end{figure}
\end{enumerate}

The geodesic of such a flat cone metric $\sigma$ is closely related to the holomorphic quadratic differential $q$. Following \cite[§5.5]{strebel1984quadratic}, a straight arc is a smooth curve along with 
$$\arg q=\theta=\text{constant}\quad 0\le\theta<2\pi$$
A maximal straight arc of $\arg q=\theta$ is called a $\theta$-\textit{trajectory}. In fact, a geodesic is exactly the concatenation of some trajectories, possibly of different $\theta$. But for a geodesic in $\Gssigma$, since it makes an angle $\pi$ at each cone point, it is composed of trajectories with the same $\theta$. In particular, if a geodesic in the surface is the projection of an element in $\Gssigma$, it must be simple, as a self-intersection implies that it contains two trajectories of different $\theta$.
  
But conversely, not all flat cone metrics come from a holomorphic quadratic differential. In the remaining of this subsection, we will use cone angle and holonomy to distinguish quadratic differential metrics among general flat cone metrics.

Recall that for a Riemannian surface, the holonomy is defined as the set of parallel transport along all closed curves with a fixed basepoint. For a flat cone surface whose cone angles are all integer multiples of $\pi$, we can still define holonomy up to $\pm\text{Id}$ in the following way: fix a nonsingular basepoint, for a nonsingular closed curve, the parallel transport is defined in the normal way; for a singular closed curve $\gamma$, find a sequence of nonsingular closed curves $\gamma_n$ approaching it, choose a convergent subsequence $P_{\gamma_n}$ and define $P_\gamma$ as the limit of $P_{\gamma_n}$.

To verify it is well-defined, it suffices to show that different subsequences give the same result. Since the metric is locally flat, the parallel transport is invariant along homotopy involving no cone points. For sufficiently large $m$ and $n$, $\gamma_m*\overline{\gamma_n}$ is either nullhomotopic or homotopic to a closed curve enclosing some cone points. For the latter case, since the cone angles are integer multiples of $\pi$, $P_{\gamma_m*\overline{\gamma_n}}=\pm\text{Id}$ and therefore $P_{\gamma_m}=\pm P_{\gamma_n}$.

\begin{thm}\label{angle and holonomy}
A flat cone metric on the surface $S$ is in $QFlat(S)$ if and only if
\begin{enumerate}
    \item[(1)] the cone angle of any cone point is of the form $k\pi$ for some integer $k\ge3$.
    \item[(2)] $\mathrm{Hol}(S) = \pm \mathrm{Id}$.
\end{enumerate}
\end{thm}

\begin{proof}
$\Rightarrow$: The cone angle is already discussed in the construction of local charts. Since the transition function of charts is $\xi' = \pm \xi + \mathrm{constant}$, $P_\gamma=\pm\text{Id}$ for nonsingular closed curve $\gamma$. Then we know from the definition of generalized holonomy that $\text{Hol}(S)=\pm\text{Id}$.

$\Leftarrow$: Fix one Euclidean chart containing a nonsingular point. Construct other Euclidean charts by requiring that the chart transition function is $\xi' = \pm \xi + \mathrm{constant}$. This construction is possible because of the holonomy condition. These Euclidean charts also give a compatible local complex coordinates. We can define a holomorphic quadratic differential form on nonsingular points, whose local form is $q = dz^2$. This quadratic form has a natural extension to cone points: if $\theta(p) = (k+2)\pi$, construct $w$ satisfying $z = \frac{2}{k+2}w^\frac{k+2}{2}$, then $w$ gives a coordinate near the cone point $p$ and $q = w^k dw^2$.
\end{proof}

\section{Chains of a Liouville current}

\subsection{Liouville current and chains}\label{computation_section}

In this subsection, we will first give a brief introduction to the Liouville current. For more details, see \cite[\textsection 8.2]{martelli2016introduction}, \cite[\textsection 4.2]{frazier2012length}.

A \textit{geodesic current} is a Radon measure on $\mathcal{G}(\tilde{S})$ which is $\pi_1(S)-$invariant. For example, given a closed curve $\gamma$ which is not nullhomotopic, the endpoints of its preimages in $\tilde{S}$ is a discrete and $\pi_1(S)-$invariant subset of $\mathcal{G}(\tilde{S})$. The counting measure on this set is exactly a geodesic current, which is still denoted by $\gamma$. 

Given a flat surface $(S,\sigma)$, the \textit{Liouville current} $L_\sigma$ is a special kind of geodesic current defined in the following way: consider the subset $\hat{S} = S \setminus \{\mathrm{cone\ points}\}$, its unit tangent bundle $T^1\hat{S}$ has a volume form which has a local expression $\frac{1}{2} dA\,d\theta$, where $dA$ is the pullback of area form of $\hat{S}$ and $d\theta$ is the pullback of angle form on the fibres. Take the interior product of this 3-form with the unit vector field generating the geodesic flow, we get a 2-form which is geodesic flow invariant. Now lift it to the universal covering $\tilde{S}$ and restrict to complete flows, its absolute value gives a $\pi_1(S)-$invariant measure on $\Gosigma$. Extend it by zero, we get a measure on the space of geodesics in $\tilde{S}$, whose support is $\Gssigma$. Finally, since the map $\partial:\Gsigma\to\Gs$ is closed by Lemma \ref{closed}, we can push forward the measure along the map and get a geodesic current $L_\sigma$ on $\mathcal{G}(\tilde{S})$, which is defined to be the Liouville current for $(S,\sigma)$. We can also see from the construction that the support of $L_\sigma$ is $\Gss$.

In \cite{bankovic2018marked}, the authors constructed chains from the Liouville current of a flat cone metric and studied the relations between chains and cone points. In the following, we will briefly recall the definition of chains and some related results.

Take two geodesics $\gamma_1,\gamma_2\in \Gsigma$, and denote the endpoints of $\gamma_i$ by $x_i, y_i$. We say that $(x_1, y_1)$ and $(x_2,y_2)$ \textit{link} if $x_1,y_1$ belong to different connected components of $S^1_\infty\setminus\{x_2,y_2\}$. It's clear that the endpoints of $\gamma_i$ link if and only if $\gamma_1,\gamma_2$ intersect transversely, or they share a saddle connection and $\gamma_1$ crosses from one side of $\gamma_2$ to the other side. For simplicity, we will call both cases "transversely intersect".

Suppose $x_2,x_1,y_1,y_2$ are in the counterclockwise order in $S^1_\infty$. Let $[p,q]\subset S^1_\infty$ be the counterclockwise interval between two points $p,q\in S^1_\infty$. For a geodesic $\gamma$ with endpoints $x,y$, we say that $\gamma$ lies between $\gamma_1$ and $\gamma_2$ if $x\in[x_2,x_1],y\in[y_1,y_2]$. Denote by $[\gamma_1, \gamma_2]$ the set of geodesics lying between $\gamma_1$ and $\gamma_2$, and by $[(x_1,y_1),(x_2,y_2)]$ the image of $[\gamma_1,\gamma_2]$ under map $\partial$. See Figure \ref{lie between}.
\begin{figure}[h]
    \centering
    \begin{tikzpicture}
    \draw (0,0) circle (2);
    \node[above] at (1,1.732) {$x_2$};
    \node[above] at (-1,1.732) {$x_1$};
    \node[below] at (-1,-1.732) {$y_1$};
    \node[below] at (1,-1.732) {$y_2$};
    
    \draw[very thick] (1,1.732) to [out=-110, in=110] (1,-1.732) (-1,1.732) to [out=-70, in=70] (-1,-1.732);
    \draw[thin] (0.52,1.93) to [out=-100,in=100] (0.52,-1.93) (-0.52,1.93) to [out=-80,in=80] (-0.52,-1.93) (0,2)--(0,-2);
    
    \node[above] at (0,2) {$[x_2,x_1]$};
    \node[below] at(0,-2) {$[y_1,y_2]$};
    
    \draw[->] (3,-0.5) to [out=70,in=-70] (3,0.5);
    
    \node[left] at(-0.8,0) {$\gamma_1$};
    \node[right] at (0.8,0) {$\gamma_2$};

    \end{tikzpicture}
    \caption{Geodesics in $[\gamma_1,\gamma_2]$.}
    \label{lie between}
\end{figure}
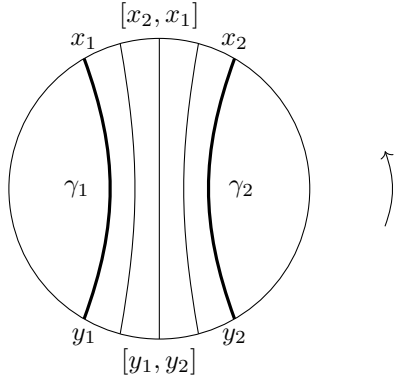

For the Liouville current $L_\sigma$ on the flat cone surface $(S,\sigma)$, a \textit{chain} is defined to be a sequence of points $\underline{x} = (\ldots, x_1,x_2,\ldots)$, $x_i\in S^1_\infty$, such that
\begin{enumerate}
    \item[(1)] $(x_i, x_{i+1})\in \supp$.
    \item[(2)] $x_i,x_{i+1},x_{i+2}$ is a counterclockwise ordered triple of distince points and 
    \[[(x_i,x_{i+1}),(x_{i+1},x_{i+2})]\cap \supp = \{(x_i,x_{i+1}), (x_{i+1},x_{i+2})\}\]
\end{enumerate}

For a chain $\underline{x}$, a sequence of geodesics $\underline{\gamma}=(\ldots,\gamma_0,\gamma_1,\ldots)$ is called a \textit{geodesic chain} of $\underline{x}$ if $\gamma_i\in\Gssigma$ and the endpoints of $\gamma_i$ are $(x_i,x_{i+1})$. The condition (2) above simply means that for a geodesic chain $\underline{\gamma}$, there is no other geodesic in $\Gssigma$ which lies between $\gamma_i$ and $\gamma_{i+1}$. The following result in \cite[Prop 4.1]{bankovic2018marked} gives the connection between chains, geodesic chains and cone points.

\begin{prop}[\cite{bankovic2018marked}]
\label{good chain prop}
Suppose $\underline{x}$ is a chain such that $\forall i$, $(x_i,x_{i+1})\in \Gss\setminus \Gts$. Then there exists a unique cone point $p$ and a geodesic chain $\underline{\gamma}=(\ldots,\gamma_0,\gamma_1,\ldots)$ of $\underline{x}$ such that $p$ is the unique cone point in each $\gamma_i$. 
\end{prop}

Notice that in Proposition \ref{good chain prop} we require every pair of adjacent points is not in $\Gts$. Such chains are called \textit{good chains} and the corresponding geodesic chains are called \textit{good geodesic chains}. From the proposition above, we know that a good geodesic chain is just a sequence of geodesic rays emanating from a fixed cone point such that the angle between two adjacent rays is $\pi$. Conversely, given a fixed cone point $p$, we can construct good chains and good geodesic chains by finding a sequence of vectors $\underline{v}=(\ldots,v_0,v_1,\ldots)$ in $T_p^1\tilde{S}$ such that
\begin{enumerate}
    \item[(1)] the angle between $v_i,v_{i+1}$ is $\pi$.
    \item[(2)] the geodesic ray $\delta_i$ with initial vector $v_i$ contains no cone point except $p$.
    \item[(3)] $(x_i,x_{i+1})\notin\Gts $ where $x_i$ is the endpoint of $\delta_i$.
\end{enumerate}
There exist only countably many sequences of vectors not satisfying condition (2) or (3), thus we could construct uncountably many good chains.

Moreover, we can use a good chain to compute the cone angle. Given a good chain $\underline{x}$ whose cone point is $p$ with cone angle $\theta(p)$, construct a function $R:\N\to \N$ in the following way: consider a point starting at $x_0$ and moving counterclockwise in $S^1_{\infty}$, let $x_{R(n)}$ be the first point it meets after passing through $x_0 $ $n$ times. There are two ways to compute the rotation angle of $x_0\to x_1 \to \ldots \to x_{R(n)}$: every time it goes back to $x_0$, it rotates for an angle $\theta(p)$, thus the total angle is $[n\theta(p),n\theta(p)+\pi]$; the other way is to count the number of geodesic rays or the number of points in the chain, which accounts for an angle $R(n)\pi$. Thus, $n\theta(p)\le R(n)\pi\le n\theta(p)+\pi$
and $$\theta(p) = \lim_{n\to\infty} \frac{R(n)}{n}\pi$$
For example, for the chain in Figure \ref{good chain}, $R(n)$ is $\frac{5n}{2}$ for even $n$ and is $\frac{5n+1}{2}$ for odd $n$, so the cone angle is $\frac{5\pi}{2}$. In particular, a periodic good chain gives a cone point whose cone angle is a rational multiple of $\pi$, while an aperiodic one corresponds to the case of irrational multiples of $\pi$. 

\begin{figure}[h]
    \centering
    \begin{tikzpicture}
    \draw (0,0) circle (2);
    \draw[->] (3.5,-0.5) to [out=70,in=-70] (3.5,0.5);
    
    \draw[thick] (0,0)--(0,2) (0,0)--(-1.17,-1.618);
    \draw[dashed] (0,0)--(-1.9,0.618) (0,0)--(1.9,0.618) (0,0)--(1.17,-1.618);
    
    \node[above] at(0,2) {$x_0$};
    \node[left] at (-1.9,0.618) {$x_3$};
    \node[below left] at (-1.17,-1.618) {$x_1(x_6)$};
    \node[below right] at (1.17,-1.618) {$x_4(x_9)$};
    \node[right] at (1.9,0.618) {$x_2(x_{12})$};
    
    \filldraw (0,0) circle (2pt);
    \end{tikzpicture}
    \caption{A chain corresponding to a cone point of cone angle $\frac{5\pi}{2}$.}
    \label{good chain}
\end{figure}
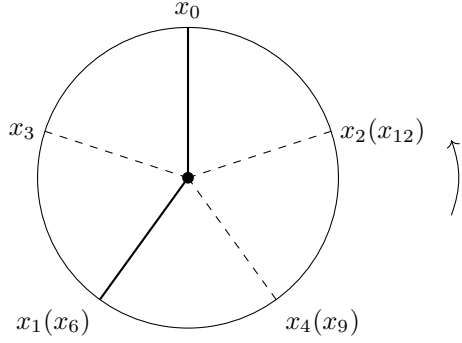

The following lemma in \cite[Lemma 4.4]{bankovic2018marked} determines whether two good chains give the same cone point. For a periodic good chain $\underline{x}$, it contains $k=k(\underline{x})$ distinct points in $S^1_\infty$, which is exactly the minimal period of this chain. Moreover, $x_i\to x_{i+1}$ generates a group of rotations, which is isomorphic to a finite abelian group. Then it must have a generator of minimal rotation angle and all other rotations are power of this generator. Therefore, there is a smallest integer $n=n(\underline{x})$ such that $x_0, x_n, x_{2n},\ldots, x_{(k-1)n}$ is the set of points in a counterclockwise order. For example, the chain $\underline{x}$ in Figure \ref{good chain} has $k(\underline{x})=5$ and $n(\underline{x})=3$.

\begin{lemma}[\cite{bankovic2018marked}]\label{interlaced}
Given two good chains $\underline{x},\underline{y}$, they give the same cone point if and only if they are in one of the following cases:
\begin{enumerate}
    \item[(1)] $\underline{x}$ and $\underline{y}$ are both periodic and $k(\underline{x})=k(\underline{y})$, $n(\underline{x})= n(\underline{y})$
    \item[(2)] $\underline{x}$ and $\underline{y}$ are both aperiodic and for any $y_i,y_{i+1}$, there exists a sequence $x_{j_n},x_{j_{n}+1}\to y_i,y_{i+1}$ as $n\to \infty$.
\end{enumerate}
\end{lemma}

Two chains (possibly not good chains) are called \textit{perfectly interlaced} if they satisfy one of the above conditions.

\subsection{Modification of good chains}

In \cite{bankovic2018marked}, the authors showed that if the geodesic currents of two flat cone metrics are the same, then the set of good chains are automatically the same and we could build an isometry between two metrics. The crucial point is whether cone angles are the same for two metrics, so there is no need to compute the cone angle.

But in our case, we need to know whether the cone angle is an integer multiple of $\pi$. In the previous section, we already saw how to compute cone angles from good chains. But it's not easy to distinguish good chains among all chains because the Liouville current does not directly give the information about $\Gts$. We turn to consider more general chains. We say a chain $\underline{x}$ is \textit{approximated} by a sequence of chains $\underline{x}^{(k)}, k\in\N$ if for any $i$, the sequence $x^{(k)}_i,x^{(k)}_{i+1}$ converges to $x_i,x_{i+1}$ as $k\to\infty$. A chain is called \textit{approachable} if it can be approximated by uncountably many distinct sequences of chains such that any two chains (possibly in different sequences) are perfectly interlaced. 

In \textsection\ref{computation_section}, we talk about how to construct good chains from a given cone point, and this method could also be used to contruct a sequnce of approximating chains. Moreover, we can construct uncountably many such sequences, so good chains are approachable.

\begin{lemma}\label{approachable}
For an approachable chain $\underline{x}$, if we apply the computation method in \textsection\ref{computation_section}, then $\lim_{n\to\infty}\frac{R(n)}{n}\pi$ still gives the cone angle of a cone point.
\end{lemma}
\begin{proof}
Recall that the computation method is based on the property that every two adjacent geodesic rays form an angle $\pi$ at the cone point. Thus, it suffices to show that an approachable chain is also generated by a sequence of vectors $\underline{v}=(\ldots,v_0,v_1,\ldots)$ such that the angle between $v_i$ and $v_{i+1}$ is $\pi$.

Since $\underline{x}$ is approximated by uncountably many sequences and there are only countably many chains that are not good chains, there must exist an approximating sequence which only consists of good chains. By Proposition \ref{interlaced}, all these chains correspond to the same unique cone point $p$. Thus, every chain is generated by a sequence of vectors $\underline{v}^{(k)}$ in $T_p\tilde{S}$. As $k\to\infty$, $\underline{v}^{(k)}$ converges to a sequence of vectors and every adjacent pair of vectors forms an angle $\pi$. By continuity of $\partial:\Gsigma\to\Gs$, the sequence of vectors generates exactly the chain $\underline{x}$.
\end{proof}

Lemma \ref{approachable} tells us that all the information about cone angle is contained in the Liouville current, or more precisely, in $\supp=\Gss$.

\begin{definition}\label{def_integral}
A Liouville current is called \textit{integral} if all cone angles are integer multiples of $\pi$.
\end{definition}

\section{Lamination and flat cone metric}

To obtain information about holonomy, it is helpful if there exists a dense geodesic in the flat cone surface. In this section, we will study the relations between lamination and foliation and find out when a dense geodesic exists. To avoid ambiguity, we will use ``hyperbolic geodesics'' and ``flat geodesics'' to distinguish different cases if necessary.

\subsection{Lamination and principal region}

Given a closed surface $S$ with a hyperbolic metric, a \textit{geodesic lamination} $\lambda$ is a closed subset of $S$ which is a disjoint union of simple geodesics. For example, a set of finitely many disjoint simple closed geodesics is a lamination. A \textit{principal region} for a lamination $\lambda$ is a connected component of $S\setminus \lambda$. A \textit{crown} is a complete hyperbolic surface with finite area and geodesic boundary, which is homeomorphic to $\mathbb{S}^1\times [0,1]\setminus A$ for a finite set $A\subset \mathbb{S}\times \{1\}$.
\begin{figure}[ht]
    \centering
    \begin{tikzpicture}[scale=0.6]
    \draw (0,0) circle (1);
    \draw (1.414,1.414) arc [radius=2, start angle=45, end angle=315];
    \draw (1.414,1.414) to [out=-45, in=180] (6,0.1);
    \draw (1.414,-1.414) to [out=45, in=180] (6,-0.1);
    \path[pattern= north east lines] (1.414,1.414) to [out=-45, in=180] (6,0.1)--(6,-0.1) to [out=-180, in=45] (1.414,-1.414) arc [radius=2, start angle=315, end angle=45];
    \draw[fill=white] (0,0) circle (1);
    \end{tikzpicture}
    \caption{A crown set homeomorphic to an annlus with one puncture.}
\end{figure}
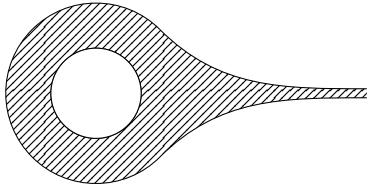

\begin{thm}[\cite{casson1988automorphisms}]\label{principal}
Let $S$ be a hyperbolic closed surface. Let $\lambda$ be a lamination without closed geodesics. Then a principal region $P$ is in one of the following forms:
\begin{enumerate}
    \item[(1)] a finite sided polygon with vertices at infinity (ideal polygon)
    \item[(2)] there exists a compact subset $P_0$ such that $P\setminus P_0$ is a union of finitely many crowns.
\end{enumerate}
\end{thm}
\begin{figure}[h]
    \centering
    \begin{tikzpicture}[scale=0.8]
    \draw (-3.268,1) arc [radius=2, start angle=30, end angle=330];
    \draw (-3.268,1) to [out=-60, in=180] (-1,0.1);
    \draw (-3.268,-1) to [out=60, in=180] (-1,-0.1);
    
    \draw (-4.293,0.707) arc [radius=1, start angle=45, end angle=315];
    \draw (-4.293,0.707) to [out=-45, in= 0] (-5.96,0.1);
    \draw (-4.293,-0.707) to [out=45, in= 0] (-5.96,-0.1);
    \draw (-6.96,0.1)-- (-9,0.1);
    \draw (-6.96,-0.1)-- (-9,-0.1);
    
    \draw[very thick] (-5,0) circle (1.5);

    \draw (3,0) circle (2);
    \draw (5,0) to [out=180,in=-135] (4.414,1.414);
    \draw (4.414,1.414) to [out=-135,in=-105] (3.5,1.93);
    \draw (3.5,1.93) to [out=-105,in=-95] (3.1743,1.992);
    \draw (3.1743,1.992) to [out=-95, in=-91] (3.035,1.9997);
    \draw (5,0) to [out=180,in=135] (4.414,-1.414);
    \draw (4.414,-1.414) to [out=135,in=105] (3.5,-1.93);
    \draw (3.5,-1.93) to [out=105,in=95] (3.1743,-1.992);
    \draw (3.1743,-1.992) to [out=95, in=91] (3.035,-1.9997);
    
    \draw (1,0) to [out=0,in=-45] (1.586,1.414);
    \draw (1.586,1.414) to [out=-45,in=-75] (2.5,1.93);
    \draw (2.5,1.93) to [out=-75,in=-85] (2.8257,1.992);
    \draw (2.8257,1.992) to [out=-85, in=-89] (2.965,1.9997);
    \draw (1,0) to [out=0,in=45] (1.586,-1.414);
    \draw (1.586,-1.414) to [out=45,in=75] (2.5,-1.93);
    \draw (2.5,-1.93) to [out=75,in=85] (2.8257,-1.992);
    \draw (2.8257,-1.992) to [out=85, in=89] (2.965,-1.9997);
    
    
    \draw[very thick] (3,2)--(3,-2);
    \node at(2.25,-0.5) {$\tilde{W}_1$};
    \node at (3.75,-0.5) {$\tilde{W}_2$};
    \end{tikzpicture}
    \captionsetup{width=0.8\textwidth}
    \caption{A principal region $P$ of the second type and its lift. Here $P_0$ is a closed geodesics and $P\setminus P_0$ is the union of two crown sets, each of which is homeomorphic to a once punctured annulus. $\tilde{W}_i$ are the corresponding lifts of crown sets. }
    \label{principal region}
\end{figure}
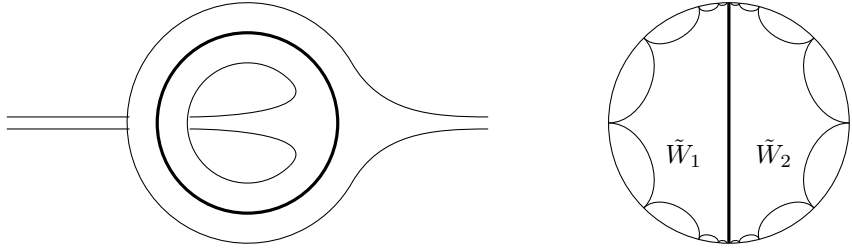
\begin{proof}
Lift $P$ to the universal covering of the surface and let $\tilde{P}$ be one of the connected components. $\tilde{P}$ is cut out by disjoint geodesics, so it's a convex subset and homeomorphic to a disk. It follows that the fundamental group of $\tilde{P}$ is trivial and thus $\tilde{P}$ is the universal covering of $P$. 

A boundary edge $\tilde{\gamma}_0$ of $\tilde{P}$ has to be adjacent to another two boundary edges $\tilde{\gamma}_{\pm1}$ which have a common endpoint with $\tilde{\gamma}_0$, otherwise $\tilde{P}$ would contain a fundamental domain of the surface $S$. Thus we obtain a sequence of geodesics $\ldots, \tilde{\gamma}_{-1}, \tilde{\gamma}_{0}, \tilde{\gamma}_{1}, \ldots$ There could be many such sequences and any two different sequences are disjoint from each other.

Suppose $\tilde{\gamma}_n=\tilde{\gamma}_0$ for one sequence of geodesics and for some integer $n$. Then this sequence cycles around and contains finitely many geodesics. Thus $\tilde{P}$ is a finite-sided ideal polygon. Set $\Gamma_P = \{f\in \pi_1(S)\,|\, f(\tilde{P}) = \tilde{P}\}$, then $P \cong \tilde{P} / \Gamma_P$. Take a point $p$ in the boundary of $P$, then its lift also lies in the boundary of $\tilde{P}$. Since the lamination $\lambda$ contains no closed geodesics, each $\tilde{\gamma}_i$ contains at most one lift of $p$. Thus, $p$ has only finitely many lifts and $\Gamma_P$ is a finite group. Each nontrivial element in $\pi_1(S)$ is of infinite order, so $\Gamma_P$ must be trivial. Thus $P$ is isometric to $\tilde{P}$ and it is also an ideal polygon.

Suppose all $\tilde{\gamma}_n$ are distinct for every sequence of geodesics. Since $P$ is of finite area, it has only finitely many boundary edges by Gauss-Bonnet Formula. But $\{\tilde{\gamma}_n\}$ is a infinite sequence, thus there exists a deck transformation $g$ fixing $\tilde{P}$ and $g. \tilde{\gamma}_0 = \tilde{\gamma}_n$ for some $n\in\Z$. Notice that as $n\to\pm\infty$, the endpoints of $\tilde{\gamma}_n$ converge to one of the endpoints of geodesic $g$ in the universal covering. Let $\tilde{W}$ be the smallest convex set containing all $\tilde{\gamma}_n$ and the geodesic representative of $g$, then $\tilde{W}$ projects down to a crown set, while $g$ descends to the closed boundary of the crown set and $\tilde{\gamma}_n$ descends to the punctured boundary.

For two different boundary edges of $\tilde{P}$, their crown sets are either the same or disjoint. Let $\tilde{P}_C$ be the union of the lifts of all crown sets in $\tilde{P}$, and set $\tilde{P}_0 = \tilde{P}\setminus \tilde{P}_C$. Then $\tilde{P}_0$ is a convex set bounded by lifts of closed geodesics. Thus $\tilde{P}_0$ projects down to a compact subset of the surface $S$, bounded by some closed geodesics.
\end{proof}

\subsection{From flat geodesics to lamination}\label{foliation_structure}

Consider a flat cone surface $(S,\sigma)$. Take a flat geodesic $h_0$ which is simple and nonsingular. Let $A$ be the closure of $h_0$. Recall that a line field is a smooth section of the projective tangent bundle. By taking the limit of a line field along $h_0$, we obtain a line field defined on $A\setminus\{\text{cone points}\}$. The complete integral curves and their limit equip $A$ with a foliation structure. Notice that a leaf might have an overlap with itself, see Figure \ref{leaf_overlap}.

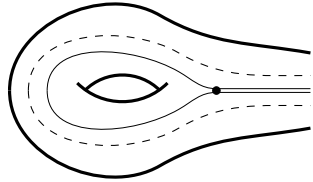
\begin{figure}[h]
    \centering
    \begin{tikzpicture}[scale=0.5]
    \draw[very thick] (-2,0) to [out=45, in=135] (0,0);
    \draw[very thick] (-2.2,0.2) to [out=-45, in=-135] (0.2,0.2);
    
    \draw[very thick] (-4,0) to [out=90, in=150] (0,2) to [out=-30, in=170] (4,1);
    \draw[very thin] (-3,0) to [out=90, in=150] (0.5,0.5) to [out=-30, in=180] (1.5,0.05)--(4,0.05);
    \draw[dashed] (-3.5,0) to [out=90, in=150] (0.5,1) to [out=-30, in=180] (4,0.4);
    
    \draw[very thick, yscale=-1] (-4,0) to [out=90, in=150] (0,2) to [out=-30, in=170] (4,1);
    \draw[very thin, yscale=-1] (-3,0) to [out=90, in=150] (0.5,0.5) to [out=-30, in=180] (1.5,0.05)--(4,0.05);
    \draw[dashed, yscale=-1] (-3.5,0) to [out=90, in=150] (0.5,1) to [out=-30, in=180] (4,0.4);
    
    \filldraw (1.5,0) circle (0.1);
    
    \end{tikzpicture}
    \caption{The singular leaf is the limit of the dashed nonsingular leaf. Although the singular leaf has a self intersection, it is not regarded as a closed leaf.}
    \label{leaf_overlap}
\end{figure}

Let $\tilde{A}$ be the preimage of $A$ in the universal covering with the induced foliation structure. If $\tilde{h}_0$ is a lift of $h_0$, $\tilde{A}=\overline{\{g.\tilde{h}_0| g\in \pi_1(S)\}}=\overline{\pi_1(S).\tilde{h}_0}$. Since $h_0$ is a nonsingular geodesic, we know all leaves must be geodesics in $\Gssigma$. Equip $S$ with a hyperbolic metric $\sigma'$ and $\tilde{S}$ with the induced hyperbolic metric. Let $\gamma(\tilde{h})$ be the unique hyperbolic geodesic connecting the endpoints of $\tilde{h}$ and $\gamma(\tilde{A})$ be the union of all such geodesics. Note that $\gamma(\tilde{A})$ is $\pi_1(S)-$invariant, thus it descends to a union of geodesics in the hyperbolic surface $(S,\sigma')$, denoted by $\gamma(A)$. Denote by $\gamma(h)$ the projection of $\gamma(\tilde{h})$.

\begin{prop}
$\gamma(A)$ is a lamination.
\end{prop}
\begin{proof}
Take any two leaf $h,h'$ in $A$. Since they have no transverse intersection, the endpoints of their lifts in the universal covering will not link. Thus, $\gamma(h)$ and $\gamma(h')$ are disjoint geodesics. Similarly, if we take $h=h'$, we can see that $\gamma(h)$ is a simple geodesic.

It remains to show that $\gamma(A)$ is a closed subset. Note that the set of leaves in $\tilde{A}$ is a closed subset of $\Gsigma$. Apply Lemma \ref{closed}, the set of endpoints is a closed subset of $\Gs$. For hyperbolic metric $\sigma'$, the map $\partial:\mathcal{G}(\sigma')\rightarrow \Gs$ is a homeomorphism, so $\gamma(\tilde{A})$ is a closed subset. Since $\gamma(\tilde{A})$ is the preimage of $\gamma(A)$ under the covering map, $\gamma(A)$ is also closed.
\end{proof}

Recall that a boundary point of a subset $S\subset T$ is a point in $S$ such that every neighborhood containing this point has nonempty intersection with $T\setminus S$. Now consider the subset $\tilde{A}\subset\tilde{S}$. A leaf $\tilde{h}\in \tilde{A}$ is called a \textit{boundary leaf} if there exists a nonsingular boundary point in $\tilde{h}$. Note that we only require the existence of such points because for cone-asymptotic leaves, it may happen that some points are boundary points, but other points of the same leaf are not boundary points. Similarly, we can define the \textit{boundary geodesic} of a geodesic lamination.

\begin{lemma}\label{boundary leaf}
If $\tilde{h}$ is a boundary leaf of $\tilde{A}$ and it's not in a flat strip, then $\gamma(\tilde{h})$ is a boundary geodesic of the lamination $\gamma(\tilde{A})$.
\end{lemma}
\begin{proof}
Suppose $\gamma(\tilde{h})$ is not a boundary geodesic. Then there exists two sequences of leaves in $\tilde{A}$, denoted by $\tilde{h}_i^P$ and $\tilde{h}_i^N$, such that $\gamma(\tilde{h}_i^P)$ converges to $\gamma(\tilde{h})$ from the positive side and $\gamma(\tilde{h}_i^N)$ converges from the negative side. The corresponding sequences of endpoints also converge to the endpoints of $\tilde{h}$. By Lemma \ref{closed}, there exists a subsequence of $\tilde{h}_i^P$ converging to a leaf which has the same endpoints as $\tilde{h}$. Since $\tilde{h}$ does not lie in a flat strip, apply Corollary \ref{flat_strip} and we can see that the limit leaf is exactly $\tilde{h}$. Similarly, there is another sequence of leaves converging from the negative side. Thus $\tilde{h}$ cannot be a boundary leaf. Contradiction.
\end{proof}
\begin{rmk}
In general, the converse of this lemma is not true.
\end{rmk}

\subsection{Totally simple Liouville current}

\begin{definition}\label{def_totally_simple}
A Liouville current $L_\sigma$ of a flat cone surface $(S,\sigma)$ is called \textit{totally simple} if for any two linked pairs of endpoints $(x_1,y_1)$, $(x_2,y_2)$ in $\supp$, $\overline{\pi_1(S).(x_1,y_1) }\cap \overline{\pi_1(S).(x_2,y_2) }=\emptyset$.
\end{definition}

Given a flat cone surface $(S,\sigma)$, if the associated Liouville current is totally simple, then every geodesic in $\Gssigma$ must descend to a simple geodesic. Otherwise, there exist two transversely intersected lifts of this nonsimple geodesic and their $\pi_1(S)$ orbits have nonempty intersection.

Now take a geodesic in $\Gssigma$. Since it descends to a simple geodesic, the closure of its $\pi_1(S)$ orbit admits a foliation structure. Moreover, if two geodesics in $\Gssigma$ intersect transversely, then their corresponding foliated sets cannot have common leaves by the definition of total simplicity.

\begin{lemma}\label{strip}
For a totally simple Liouville current, every flat strip $\tilde{\Lambda}$ in $\tilde{S}$ must descend to an annulus in $S$. The same holds true for a flat half strip. In particular, the boundary geodesic of a (half) flat strip must descend to a closed geodesic.
\end{lemma}
\begin{proof}
Denote the projection of $\tilde{\Lambda}$ by $\Lambda$. $\Lambda$ must have a self-intersection, otherwise it would be of infinite area. Note that the boundary of $\tilde{\Lambda}$ is a geodesic in $\Gssigma$, we know from the discussion above that it projects to a simple geodesic. Therefore, $\Lambda$ could not have a transverse self-intersection. If $\Lambda$ is not an annulus, then the only possibility is that the strip has an overlap with itself, see Figure \ref{horizontal translation}. Now consider all the lifts of this overlapped strip. Every single lift has an overlap with another two lifts. Thus, the preimage, which is also the union of all lifts, is isometric to an Euclidean plane. For a flat half strip, we can similarly obtain an Euclidean half plane. In both cases the preimage is of infinite diameter and it would contain the fundamental domain of the surface. It implies that the surface is Euclidean and contains no cone point, which is impossible. Therefore, $\Lambda$ must be an annulus.
\end{proof}
\begin{figure}[h]
    \centering
    \begin{tikzpicture}

    \draw[->] (-2,0) to [out=90,in=90] (2,0) to [out=-90,in=-90] (-1.5,0);
    \draw[->] (-1,0) to [out=90,in=90] (1,0) to [out=-90,in=-90] (-0.5,0);
    
    \path[pattern=north east lines, pattern color=gray] (-1,0)--(-2,0) to [out=90,in=90] (2,0) to [out=-90,in=-90] (-1.5,0)--(-0.5,0) to [out=-90,in=-90] (1,0) to [out=90,in=90] (-1,0);
    
    \node[below] at(-1.75,0.1) {$\underbrace{\, }$};
    \draw[->] (-1.75,-0.25)--(-2,-0.8);
    \node[align=center, below] at(-2,-0.8) {overlap};
    \end{tikzpicture}
    \caption{Self-intersection of a flat strip with overlapping.}
    \label{horizontal translation}
\end{figure}
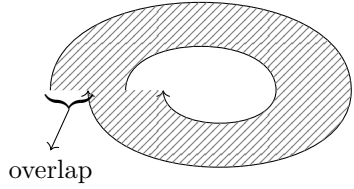

\begin{prop}\label{nonsingular}
For a flat cone surface $(S,\sigma)$ whose cone angles are integer multiples of $\pi$, if the Liouville current is totally simple, then there exists a geodesic $h$ in $S$ such that it is simple, nonsingular, not closed, and every leaf in the closure of $h$ contains at most one cone point.
\end{prop}
\begin{proof}
Fix a nonsingular point $p$ in the universal covering $\tilde{S}$, there are uncountably many geodesic in $\Gssigma$ passting through $p$. Thus it suffices to show that there are at most countably many geodesics satisfying one of the following conditions:
\begin{enumerate}
    \item[(1)] it contains at least one cone point
    \item[(2)] it descends to a closed geodesic
    \item[(3)] one leaf of the foliation is in $\Gts$.
\end{enumerate}

For (1), use the fact that there are only countably many cone points and countably many geodesics in $\Gts$. For (2), if a geodesic descends to a closed geodesic, then it must connect $p$ and $g.p$ for some $g\in\pi_1(S)$. Thus, there are only countably many such geodesics. For (3), different geodesics passing through $p$ give different foliated sets, and there is no common leaf for any two of such foliated sets. Since there are only countably many geodesics in $\Gts$, geodesics satisfying (3) are at most countable.
\end{proof}

\subsection{Existence of a dense flat geodesic}

Take a flat geodesic $h$ satisfying requirements in Proposition \ref{nonsingular}, let $A$ be its closure with the foliation structure. We want to apply Theorem \ref{principal} to $\gamma(A)$ to get some information about boundary geodesics. But first of all we need to show there are no closed leaves in $\gamma(A)$ or in $A$. Suppose $j$ is a closed flat geodesic. Since it could be approximated by the nonsingular geodesic $h$, the lift $\tilde{j}$ of $j$ must be in $\Gssigma$. Recall that in Proposition \ref{saddle}, we give the equivalent condition for a geodesic lying in $\Gssigma$. Now in this case, the side where the angle is $\pi$ at cone points cannot switch, otherwise there would be a transverse intersection of $\tilde{j}$ and some lift of $h$. Then there exists an open flat strip on the side of $\tilde{j}$ where all angles are $\pi$, because in the surface $S$, the minimal distance between $j$ and a cone point not in $j$ is strictly positive. Since $\tilde{j}$ is approximated by lifts of $h$, one of the lifts must intersect this flat strip and it must be parallel to $\tilde{j}$. But geodesics parallel to $\tilde{j}$ will descend to a closed geodesic. Contradiction.

\begin{thm}\label{dense}
If the Liouville current is totally simple, then there exists a nonsingular simple flat geodesic whose closure is the entire surface $S$. 
\end{thm}
\begin{proof}
Using Proposition \ref{nonsingular}, we can find a nonsingular nonclosed flat geodesic $h$ such that all leaves in its closure $A$ contain at most one cone point. In the above discussion, we know that the corresponding geodesic lamination contains no closed geodesic and we can apply Theorem \ref{principal}

We claim that all principal regions of $\gamma(A)$ are ideal polygons. Otherwise, suppose $P$ is a principal region which is the union of a compact set and some crown sets. Take a lift of one crown set. As in the proof of Theorem \ref{principal}, let $\{\ldots,\tilde{\gamma}_0,\tilde{\gamma}_1,\ldots\}$ be the sequence of adjacent boundary edges and $\tilde{\gamma}_n = g.\tilde{\gamma}_0$ for some $g\in\pi_1(S)$. Note that each $\tilde{\gamma}_i$ comes from a leave, denoted by $\tilde{h}_i$. $\tilde{h}_i$ and $\tilde{h}_{i+1}$ are asymptotic, so they must be in one of the three cases in Lemma \ref{asymptotic rays}. But $\tilde{h}_i$ and $\tilde{h}_{i+1}$ cannot bound a flat half-strip by Lemma \ref{strip}, and each of them contains at most one cone point, so every adjacent pair of leaves must be cone-asymptotic and share the unique cone point. Consequently, $\tilde{h}_0$ and $\tilde{h}_n$ also share the same cone point. On the other hand, since $\tilde{\gamma}_n = g.\tilde{\gamma}_0$, $g$ must map the cone point on $\tilde{h}_0$ to the cone point on $\tilde{h}_n$. Then $g$ has a fixed point in the universal covering, which is absurd.

Now suppose $A$ is not the entire surface $S$, then the complement is the union of some ideal polygons. Since the complement of $A$ is an open set, $A$ must contain a boundary point and therefore a boundary leaf, denoted by $k$. Its lift $\tilde{k}$ is a boundary leaf of $\tilde{A}$ and it generates a boundary geodesic $\gamma(\tilde{k})$ of the lamination $\gamma(\tilde{A})$ by Lemma \ref{boundary leaf}. $\gamma(\tilde{k})$ is also a boundary edge of a principal region which is an ideal polygon. It´s clear that boundary edges of an ideal polygon come from a sequence of leaves which form a good geodesic chain. But for every nonsingular point of a leaf in a good geodesic chain, it can be approximated from both the positive and the negative side because it belongs to two adjacent leaves. See an example in Figure \ref{sides}. Thus it cannot be a boundary point and $\tilde{k}$ cannot be a boundary leaf. Contradiction.
\end{proof}

\begin{figure}[h]
    \centering
    \begin{tikzpicture}
    \draw (0,0) circle (2);
    \draw[very thick] (0,2)--(0,0)--(-2,0);
    \draw[dashed] (0,-2)--(0,0)--(2,0);
    
    \draw[color=gray,dotted, thick] (0,2) to [out=-90,in=0] (-2,0) to [out=0,in=90] (0,-2) to [out=90,in=180] (2,0) to [out=180,in=-90] (0,2);
    
    \node[above] at(0,2) {$x_0$};
    \node[left] at(-2,0) {$x_1$};
    \node[below] at (0,-2) {$x_2$};
    \node[right] at (2,0) {$x_3$};
    \node[below right] at(0,0) {$\xi$};
    \end{tikzpicture}
    \captionsetup{width=0.8\textwidth}
    \caption{The leaf $x_0\xi x_1$ is approached by nonsingular leaves from the above left side, while $x_0\xi$ is also approached from the right side as a subray of $x_0\xi x_3$.}
    \label{sides}
\end{figure}
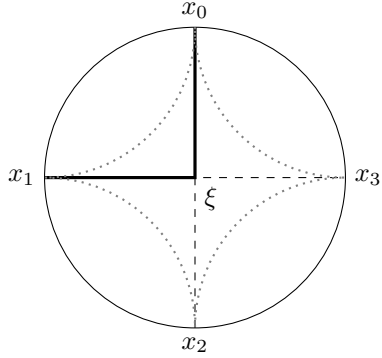

\section{Chains and holonomy}

In this section, we assume that the surface $S$ is equipped with a flat cone metric and the Liouville current is integral and totally simple. In particular, we know from the discussion in § \ref{holonomy and cone angle} that the holonomy is well-defined up to $\pm\mathrm{Id}$. We will show that the holonomy of such a surface is $\pm\mathrm{Id}$. We first define an equivalence relation on $T^1_p(S)$ for every point $p\in S$: two vectors are equivalent if and only if they make an angle $k\pi,k\in\Z$. Each equivalent class is called a \textit{direction} at $p$. To prove the holonomy is $\pm \mathrm{Id}$, it suffices to define a direction at every point which is invariant under parallel transport.

By Theorem \ref{dense}, there exists a simple nonsingular flat geodesic $h$ which is dense in the surface. It equips the surface with a foliation structure. Furthermore, it gives a field of directions by the tangent vectors of leaves. We can lift it to the universal covering and get a $\pi_1(S)$-invariant field of directions in $\tilde{S}$.

Take a geodesic arc $\beta:[0,1]\to \tilde{S}$ connecting two points $p$ and $q$ such that there is no cone point in the interior of this arc ($p$ and $q$ might be cone points). Since $\beta$ is compact, there exists a neighborhood isometric to $[0,1]\times (-\epsilon,\epsilon)$ such that $\beta$ is mapped to the $x$-axis $[0,1]\times \{0\}$. Restrict the field of directions to this arc, then the direction at the point $\beta(t)$ is uniquely determined by the angle between itself and the positive $x$-axis. Denote this angle by $\theta(t)$. The intersection of $\tilde{h}$ and this neighborhood is a set of countably many disjoint subarcs. Denote the subarc passing through $\beta(t)$ by $\tilde{h}_t$.

We claim that $\theta$ is a constant function. If $\beta$ is a subarc of $\tilde{h}$, then $\theta(t)=0$ for all $t$. Thus we only need to consider the case where $\beta$ intersects $\tilde{h}$ transversely. Suppose $\theta$ is not a constant function, then $\lim_{\epsilon\to 0} \frac{1}{\epsilon}(\theta(t_0+\epsilon)-\theta(t_0))$ is nonzero or does not exist for some $t_0$. There exists a sequence $\{\epsilon_n\}$ converging to 0 such that $\frac{1}{\epsilon_n}(\theta(t_0+\epsilon_n)-\theta(t_0))$ converges to a nonzero constant or infinity. In both cases, for sufficiently large $n$, there exists a positive constant $M$ such that $|\theta(t_0+\epsilon_n)-\theta(t_0)|>\epsilon_n M$. Without loss of generality, we may assume $\theta(t_0+\epsilon_n)<\theta(t_0)$ for all $n$.

If there is no cone point nearby, $\tilde{h}_{t_0}$ and $\tilde{h}_{t_0+\epsilon_n}$ would intersect and make an angle $\alpha_n=\theta(t_0)-\theta(t_0+\epsilon_n)$. See Figure \ref{triangle}. Let $d_n$ be the distance between the intersection point and $\beta(t_0)$, $f_n$  be the distance between the intersection point and $\beta(t_0+\epsilon_n)$. Then
\[d_n < \frac{\epsilon_n \sin \theta(t_0)}{\tan(\epsilon_nM)}-\epsilon_n\cos\theta(t_0) \quad f_n<\frac{\epsilon_n\sin \theta(t_0)}{\sin(\epsilon_nM)}
\]
Thus $\{d_n\}$ and $\{f_n\}$ are bounded sequences. Take a constant $C$ bigger than the upper bound of $\{d_n\}$ and $\{f_n\}$.
\begin{figure}[h]
    \centering
    \begin{tikzpicture}[scale=0.6]
    \draw[very thick] (-3,0)--(4,0);
    \draw[dashed] (0,2)--(-1,0) (4,1)--(2,0) (-2.5,-3)--(-2,-2)--(-3,-2.5);
    \draw (-1,0)--(-2,-2) (2,0)--(-2,-2);
    \draw (-1.82,-1.91) arc [radius=0.2, start angle=26.565, end angle=63.435 ];
    \draw (-0.8,0) arc [radius=0.2, start angle=0, end angle=63.435];
    
    \node[above] at(-3,0) {$\beta$};
    \node[above right] at(-1.82,-1.7) {$\alpha_n$};
    \node[above right] at(-0.8,0) {$\theta(t_0)$};
    \node[left] at (-0.3,1.9) {$\tilde{h}_{t_0}$};
    \node[above] at (4,1) {$\tilde{h}_{t_0+\epsilon_n}$};
    \node[left] at(-1.5,-0.9) {$d_n$};
    \node[right] at (0,-1) {$f_n$};
    \end{tikzpicture}
    \caption{The intersection of two arcs without cone points.}
    \label{triangle}
\end{figure}
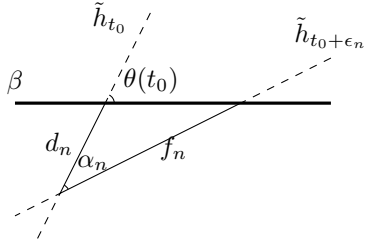

Since $\tilde{h}_{t_0}$ and $\tilde{h}_{t_0+\epsilon_0}$ are subarcs of the same simple flat geodesic, there must be a cone point to avoid intersection. Let $A$ be the point on $\tilde{h}_{t_0}$ such that the distance between $A$ and $\beta(t_0)$ is $C$, $B$ be the point on $\tilde{h}_{t_0+\epsilon_0}$ such that the distance between $B$ and $\beta(t_0+\epsilon_0)$ is $C$, $D = \beta(t_0)$, $C = \beta(t_0+ \epsilon_0)$. The cone point, denoted by $\xi_0$, must lie in the quadrilateral $ABCD$. Similarly, for sufficient large $n$, $\tilde{h}_{t_0}$ and $\tilde{h}_{t_0+\epsilon_n}$ form a more narrow quadrilateral $AB'C'D$ which does not contain $\xi_0$, where $B'C'$ is a geodesic segment in $\tilde{h}_{t_0+\epsilon_n}$. Then there must be a new cone point $\xi_1$ in $AB'C'D$, see Figure \ref{quadrilateral}. Repeat this process, we could find infinitely many cone points in the compact region $\{p\in \tilde{S}\mid{}d(p,\beta)\le C\}$, which contradicts to the fact that cone points are discrete.

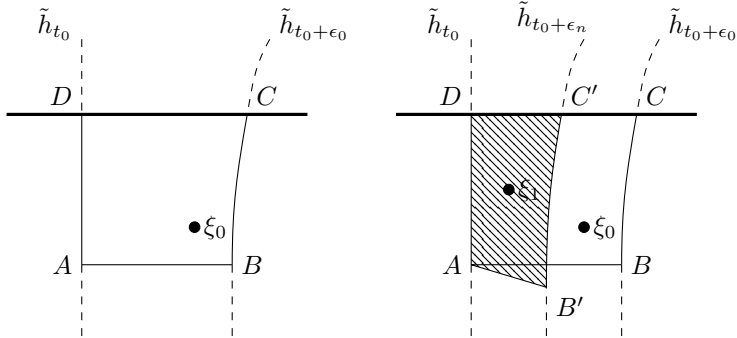
\begin{figure}[h]
    \centering
    \begin{subfigure}{0.4\textwidth}
    \begin{tikzpicture}
    \draw[very thick] (-1,0)--(3,0);
    \draw[dashed] (0,1)--(0,0) (0,-2)--(0,-3) (2,-2)--(2,-3)  (2.5,1) to [out=-120, in=80] (2.2,0);
    \draw (0,0)--(0,-2) (2.2,0) to [out=-100, in=90] (2,-2);
    \draw (0,-2)--(2,-2);
    
    
    \filldraw (1.5,-1.5) circle (2pt);
    
    \node[above left] at(0,0) {$D$};
    \node[left] at (0,-2) {$A$};
    \node[right] at (2,-2) {$B$};
    \node[above right] at (2.2,0) {$C$};
    \node[right] at (1.5,-1.5) {$\xi_0$};
    \node[left] at (0,1.2) {$\tilde{h}_{t_0}$};
    \node[right] at (2.5,1.2) {$\tilde{h}_{t_0+\epsilon_0}$};
    \end{tikzpicture}
    \end{subfigure}
    \begin{subfigure}{0.4\textwidth}
    \begin{tikzpicture}
    \draw[very thick] (-1,0)--(3,0);
    \draw[dashed] (0,1)--(0,0) (0,-2)--(0,-3) (1,-2)--(1,-3) (2,-2)--(2,-3) (1.5,1) to [out=-120, in=80] (1.2,0) (2.5,1) to [out=-120, in=80] (2.2,0);
    \draw (0,0)--(0,-2) (1.2,0) to [out=-100, in=90] (1,-2.3) (2.2,0) to [out=-100, in=90] (2,-2);
    \draw (0,-2)--(2,-2) (0,-2)--(1,-2.3);
    
    \path[pattern=north west lines] (0,0)--(1.2,0) to [out=-100, in=90] (1,-2.3)--(0,-2)--(0,0);
    
    \filldraw (0.5,-1) circle (2pt) (1.5,-1.5) circle (2pt);
    
    \node[above left] at(0,0) {$D$};
    \node[left] at (0,-2) {$A$};
    \node[right] at (2,-2) {$B$};
    \node[above right] at (2.2,0) {$C$};
    \node[below right] at (1,-2.3) {$B'$};
    \node[above right] at (1.2,0) {$C'$};
    \node[right] at (0.5,-1) {$\xi_1$};
    \node[right] at (1.5,-1.5) {$\xi_0$};
    \node[left] at (0,1.2) {$\tilde{h}_{t_0}$};
    \node[right] at (2.5,1.2) {$\tilde{h}_{t_0+\epsilon_0}$};
    \node[left] at (1.7,1.3) {$\tilde{h}_{t_0+\epsilon_n}$};
    \end{tikzpicture}
    \end{subfigure}
    \caption{Quadrilaterals and cone points.}
    \label{quadrilateral}
\end{figure}

Now take two points $p,q$ and consider the parallel transport along some curve joining these two points. Since the flat cone metric is locally flat, we may assume the curve is a concatenation of geodesic arcs. We already know that the angle of the direction is constant along a single arc, thus the parallel transport indeed maps the direction at $p$ to the direction at $q$. In conclusion, the field of direction constructed above is invariant under parallel transport.

Finally we could prove our main theorem:
\begin{thm}\label{main theorem}
Let $(S,\sigma)$ be a closed oriented flat cone surface of genus at least 2. Then $\sigma$ lies in QFlat(S) if and only if the Liouville current is integral and totally simple.
\end{thm}
\begin{proof}
We already prove that if the Liouville current is integral and totally simple, then it is a metric in $QFlat(S)$. For the other direction, if the metric $\sigma$ comes from a holomorphic quadratic differential $q$, it is obvious that the Liouville current must be integral by Theorem \ref{angle and holonomy}. Moreover, a geodesic in $\Gssigma$ is composed of a sequence of trajectories for the same angle $\theta$ and it must be simple. There exists a foliation structure on its closure and every leaf must have the same angle $\theta$. For any two linked pairs of endpoints $(x_1,y_1)$ and $(x_2,y_2)$ in $\Gss$, they correspond to trajectories of different angles and $\overline{\pi_1(S).(x_1,y_1) }\cap \overline{\pi_1(S).(x_2,y_2) }$ must be empty. Thus the Liouville current is totally simple.
\end{proof}

From Theorem \ref{nonsingular} and Theorem \ref{main theorem}, we automatically obtain the following interesting corollary.
\begin{cor}\label{main corollary}
    Given a surface with a quadratic differential and any point $p$, there must exist a nonsingular geodesic passing through $p$ which is dense in the surface. Moreover, if we parameterize the space of nonsingular geodesics passing through $p$ by the angle of tangent vector at $p$, then almost every such geodesic will be dense in the surface.
\end{cor}

\bibliography{ref}

\begin{thebibliography}{CCB88}

\bibitem[Bal85]{ballmann1985manifolds}
Werner Ballmann.
\newblock Manifolds of non positive curvature.
\newblock In {\em Arbeitstagung Bonn 1984}, pages 261--268. Springer, 1985.

\bibitem[BH13]{bridson2013metric}
Martin~R Bridson and Andr{\'e} Haefliger.
\newblock {\em Metric spaces of non-positive curvature}, volume 319.
\newblock Springer Science \& Business Media, 2013.

\bibitem[BL18]{bankovic2018marked}
Anja Bankovic and Christopher Leininger.
\newblock Marked-length-spectral rigidity for flat metrics.
\newblock {\em Transactions of the American Mathematical Society},
  370(3):1867--1884, 2018.

\bibitem[Bon88]{bonahon1988geometry}
Francis Bonahon.
\newblock The geometry of teichm{\"u}ller space via geodesic currents.
\newblock {\em Inventiones mathematicae}, 92(1):139--162, 1988.

\bibitem[CCB88]{casson1988automorphisms}
Andrew~J Casson, Andrew~J Casson, and Steven~A Bleiler.
\newblock {\em Automorphisms of surfaces after Nielsen and Thurston}.
\newblock Number~9. Cambridge University Press, 1988.

\bibitem[DLR10]{duchin2010length}
Moon Duchin, Christopher~J Leininger, and Kasra Rafi.
\newblock Length spectra and degeneration of flat metrics.
\newblock {\em Inventiones mathematicae}, 182(2):231--277, 2010.

\bibitem[Fra12]{frazier2012length}
Jeffrey Frazier.
\newblock {\em Length spectral rigidity of non-positively curved surfaces}.
\newblock University of Maryland, College Park, 2012.

\bibitem[HP97]{hersonsky1997rigidity}
Sa’ar Hersonsky and Fr{\'e}d{\'e}ric Paulin.
\newblock On the rigidity of discrete isometry groups of negatively curved
  spaces.
\newblock {\em Commentarii Mathematici Helvetici}, 72(3):349--388, 1997.

\bibitem[Mar16]{martelli2016introduction}
Bruno Martelli.
\newblock An introduction to geometric topology.
\newblock {\em arXiv preprint arXiv:1610.02592}, 2016.

\bibitem[Ota90]{otal1990spectre}
Jean-Pierre Otal.
\newblock Le spectre marqu{\'e} des longueurs des surfaces {\`a} courbure
  n{\'e}gative.
\newblock {\em Annals of Mathematics}, 131(1):151--162, 1990.

\bibitem[Str84]{strebel1984quadratic}
Kurt Strebel.
\newblock Quadratic differentials.
\newblock In {\em Quadratic Differentials}, pages 16--26. Springer, 1984.

\end{thebibliography}

\bibliographystyle{alpha}


\end{document}